\documentclass[10pt,a4paper]{article}
\usepackage[latin1]{inputenc}
\usepackage[english]{babel}
\usepackage{amsmath}
\usepackage{amsfonts}
\usepackage{amssymb}
\usepackage{array}
\usepackage{hyperref}
\usepackage{bbold}
\usepackage{ntheorem}
\newtheorem*{theorem-non}{Theorem}
\usepackage{graphicx}
\usepackage{caption}
\usepackage{subcaption}
\usepackage{epstopdf}
\newtheorem{theorem}{Theorem}[section]
\newtheorem{corollary}[theorem]{Corollary}

\newtheorem{lemma}[theorem]{Lemma}
\newtheorem{proposition}[theorem]{Proposition}
\newtheorem*{proof}{\bf{Proof}}
\newtheorem{remark}{Remark}
\newtheorem{definition}{Definition}
\title{H\"{o}lderian convergence of fractional extended nabla operator to fractional derivative}
\author{L. Khitri-Kazi-Tani\thanks{corresponding author: kazitani.leila13@gmail.com}
\and H. Dib 
\thanks{h\_dib@mail.com}}
\date{}
\AtEndDocument{
  \textit{\\\\L. Khitri-Kazi-Tani\\Laboratoire de Statistiques et mod\'elisation stochastique\\
  Department of Mathematics,\\ Abou Bekr Belkaid University, Tlemcen,\\ Algeria\\e-mail: kazitani.leila13@gmail.com 
  \\\\ H. Dib\\Laboratoire de Statistiques et mod\'elisation stochastique\\Department of Mathematics,\\ Abou Bekr Belkaid University, Tlemcen,\\ Algeria\\e-mail: h\_dib@mail.com } }
\providecommand{\keywords}[1]{\textbf{\textit{Keywords.}} #1}

\begin{document}

\maketitle

\begin{abstract}
In this paper, we construct the fractional extended nabla operator as fractional power of linear spline of backward difference operator. Then we prove the strong convergence of this  operator to fractional derivative in a H\"{o}lder space setting. Finally numerical examples are presented.
\end{abstract}
\section{Introduction}

\qquad It is well-known that fractional calculus is a developing field
both from the theoretical and applied point of view. The fractional
differential equations turned out to be the best tool for modeling
memory-dependent processes \cite{das}. We refer to the monograph \cite{zhoo},
which contains almost complete qualitative fractional differential equation
theory, and to the monograph \cite{diethelm} \ for an application oriented exposition.

Besides this rapid development, the notion of difference operators has been
extended to fractional calculus in different ways \cite{hilfer},
\cite{orteguiera}, \cite{samko93}. The discrete calculus provides a natural
setting to define such operators. However, in literature there is no single
definition of fractional difference operators and this situation can be
confusing (see for example \cite{abdelj2011},\cite{abdelj2012} and \cite{mozy}).

Another way to define this operators is to consider the fractional power of
positive discrete operators see \cite{ash}.

Effectively, functional calculus is a consistent way to define operators of the
form $A^{\alpha}$ for a given linear operator $A$ in a Banach space$.$ The
fundamental aspects of the theory of fractional powers of non-negative
operators are given in \cite{martinez}. Sectorial operators satisfy a
resolvent condition that leads to define the fractional power of such
operators. The functional calculus for sectorial operators has been developed
by M. Haase in the book \cite{haase}.

Apart from \cite{ash}, we do not know about any other work done on fractional
difference derivative in terms of spectral operator theory.

In this paper, we define the fractional difference as fractional power of the
nabla operator in a H\"{o}lder space. The H\"{o}lder spaces offer an
interesting point of view in the analysis of fractional integrals and
derivatives. This framework was developed by Samko et al. for fractional
operators in the sense of Marchaud \cite{samko93},\cite{samko94}. In this
functional framework we study the strong convergence of the extended backward
differences to the derivative. We construct the fractional operators
associated  and the strong convergence result is proved. Lastly, some examples
are provided to show the effectiveness of the approach.

This paper is organized as follows: The section 2 is devoted to preliminaries
and some H\"{o}lderians tools. Then, the operators in this context are
defined. In section 3, we give the basic definitions and results for the
fractional power of sectorial operator and we construct the different
fractional operators as fractional power of sectorial operator in H\"{o}lder
spaces setting. In section 4, we discuss the strong convergence of the
operators involved. Some examples are given in section 5.

\section{Preliminaries on operators in H\"{o}lder spaces}

Without loss of generality, we assume that the functions are defined on the
interval $\left[  0,1\right]  $. Let $\ H^{\beta}$ be the Banach space of
H\"{o}lderian function on $\left[  0,1\right]  $ with exponent $\beta$, where
$0<\beta<1$ and such that $f(0)=0,$ endowed with the norm $\left\Vert
f\right\Vert _{\beta}=\omega_{\beta}(f,1),$ where%

\[
\omega_{\beta}(f,\delta)=\underset{0<\left\vert s-t\right\vert \leq\delta
}{\underset{s,t\in\left[  0,1\right]  }{\sup}}\frac{\left\vert
f(t)-f(s)\right\vert }{\left\vert t-s\right\vert ^{\beta}}
\]

Let $H_{0}^{\beta}$ be the subspace defined by%

\[
H_{0}^{\beta}=\left\{  f\in H^{\beta},\underset{\delta\rightarrow0}{\lim
}\omega_{\beta}(f,\delta)=0\right\}
\]

\begin{remark}
If $f\in H_{0}^{\beta}$ then $\left\vert f(t)-f(t-h)\right\vert =o(h^{\beta})$
uniformly in $t,$ for $t=h$ we get $\left\vert f(h)\right\vert =o(h^{\beta}).$
\end{remark}

\begin{remark}
If $f\in H^{\beta}$ then $\left\Vert f\right\Vert _{\infty}\leq\left\Vert
f\right\Vert _{\beta}.$ Indeed, for every $x\in\left]  0,1\right]  ,$\\
$\left\vert f(x)\right\vert =\dfrac{\left\vert f(x)-f(0)\right\vert }%
{x^{\beta}}x^{\beta}\leq\left\Vert f\right\Vert _{\beta}.$
\end{remark}

\begin{remark}
If $f^{\prime}\in H^{\beta}$ then $\omega_{\beta}\left(  f,h\right)  \ $tends
to $0$ as $h\ $tends to $0.$

Indeed, for every $x$ $,y\in\left[  0,1\right]  $, $x\neq y$ there exists some
$\xi$ $\in$ $\left]  x,y\right[  $ such that%

\[
\dfrac{\left\vert f(x)-f(y)\right\vert }{\left\vert x-y\right\vert ^{\beta}%
}=\left\vert x-y\right\vert ^{1-\beta}\left\vert f^{\prime}(\xi)\right\vert
\leq\left\vert x-y\right\vert ^{1-\beta}\left\Vert f'\right\Vert _{\beta}%
\]

which leads to%

\begin{equation}
\omega_{\beta}\left(  f,h\right)  \leq h^{1-\beta}\left\Vert f^{\prime
}\right\Vert _{\beta}\label{majornormf2}%
\end{equation}

\end{remark}

The H\"{o}lder norm of the piecewise linear interpolation is given by the next
lemma. As far as we know, this result was first proved by H. E. White, Jr, in a
general setting see \cite[3.2 Corollary p 106]{white} but we follow
\cite{rackauska} in the presentation.

\begin{lemma}
[see lemma 3.1 \cite{rackauska}]Let $t_{0}=0<t_{1}<\cdots<t_{n}=1$ be a
partition of $\left[  0,1\right]  $ and $f$ be a real valued polygonal line
function on $\left[  0,1\right]  $ with vertices at $t_{i}$'$s,$ i.e. $f$ is
continuous on $\left[  0,1\right]  $ and its restriction to each interval
$\left[  t_{i},t_{i+1}\right]  $ is an affine function. Then for any
$0\leq\beta<1,$%

\[
\underset{0\leq s<t\leq1}{\sup}\frac{\left\vert f(t)-f(s)\right\vert }{\left(
t-s\right)  ^{\beta}}=\underset{0\leq i<j\leq1}{\max}\frac{\left\vert
f(t_{j})-f(t_{i})\right\vert }{\left(  t_{j}-t_{i}\right)  ^{\beta}}%
\]

\label{rac}
\end{lemma}

\begin{definition}
For $0<h<1$ fixed, let $\Delta_{h}$ be the subdivision of $\left[  0,1\right]  $ \ in $n$
subintervals with $n=\left[  1/h\right]  $ and $t_{k}=kh,$ for each
$k=0,1,\ldots,n$ where $\left[  a\right]  $ means the integer part of $a.$
We denote by $\mathcal{I}_{h}\in\mathcal{L(}%
H\mathcal{^{\beta})}$ the piecewise linear interpolation operator defined by%
\[
\left(  \mathcal{I}_{h}f\right)  (x):=\sum_{k=1}^{n}\left(  \frac{x-t_{k-1}}{h}f(t_{k})+\frac{t_{k}-x}{h}f(t_{k-1})\right)  \mathbb{1}_{\left[  t_{
k-1},t_{k}\right]  }(x)%
\]

\end{definition}

In the following lemma the remainder of piecewise linear interpolation is
expressed in H\"{o}lder norm.

\begin{lemma}
Let $(r_{h}f)(x)=\left(  I-\mathcal{I}_{h}\right)  f(x)$ then%

\[
\left\Vert (r_{h}f)\right\Vert _{\beta}\leq4\omega_{\beta}(f,h)\text{ }%
\]
\label{erreurinterpolation}
\end{lemma}

\begin{proof}
First let us suppose that $x,y\in\left]  t_{k-1},t_{k}\right]  $ then%

\[
(r_{h}f)(x)-(r_{h}f)(y)=f(x)-f(y)-\frac{x-y}{h}\left(  f(t_{k})-f(t_{k-1}%
)\right)
\]

It follows, from $\left\vert x-y\right\vert <h$ that%

\begin{align*}
\dfrac{\left\vert (r_{h}f)(x)-(r_{h}f)(y)\right\vert }{\left\vert
x-y\right\vert ^{\beta}}  & \leq\dfrac{\left\vert f(x)-f(y)\right\vert
}{\left\vert x-y\right\vert ^{\beta}}+\dfrac{\left\vert f(t_{k})-f(t_{k-1}%
)\right\vert }{h^{\beta}}\\
& \leq2\omega_{\beta}(f,h)
\end{align*}

Second, suppose that $x\in\left]  t_{k-1},t_{k}\right]  $ and $y\in\left]
t_{k},t_{k+1}\right]  $ then from the first case%

\begin{align*}
\dfrac{\left\vert (r_{h}f)(x)-(r_{h}f)(y)\right\vert }{\left\vert
x-y\right\vert ^{\beta}}  & \leq\dfrac{\left\vert (r_{h}f)(x)-(r_{h}%
f)(t_{k})\right\vert }{\left\vert x-t_{k}\right\vert ^{\beta}}+\dfrac
{\left\vert (r_{h}f)(t_{k})-(r_{h}f)(y)\right\vert }{\left\vert t_{k}%
-y\right\vert ^{\beta}}\\
& \leq4\omega_{\beta}(f,h)
\end{align*}

Third, suppose that $x\in\left]  t_{k-1},t_{k}\right]  $ and $y\in\left]
t_{m-1},t_{m}\right]  $ with $\left\vert x-y\right\vert >h$ then%

\begin{align*}
& \dfrac{\left\vert (r_{h}f)(x)-(r_{h}f)(y)\right\vert }{\left\vert
x-y\right\vert ^{\beta}}\\
& \leq\dfrac{\left\vert (r_{h}f)(x)-(r_{h}f)(t_{k})\right\vert }{\left\vert
x-y\right\vert ^{\beta}}+\dfrac{\left\vert (r_{h}f)(t_{k})-(r_{h}%
f)(t_{m-1})\right\vert }{\left\vert x-y\right\vert ^{\beta}}+\dfrac{\left\vert
(r_{h}f)(t_{m-1})-(r_{h}f)(y)\right\vert }{\left\vert x-y\right\vert ^{\beta}}%
\end{align*}
Knowing that $(r_{h}f)(t_{k})=(r_{h}f)(t_{m-1})=0$ then%
\[
\dfrac{\left\vert (r_{h}f)(x)-(r_{h}f)(y)\right\vert }{\left\vert
x-y\right\vert ^{\beta}}\leq4\omega_{\beta}(f,h)
\]

\end{proof}

We denote by $A=\dfrac{d}{dx}$ the differential operator acting on $H^{\beta}$
with domain,%

\[
D(A)=\left\{  f\in H^{\beta},f^{\prime}\in H^{\beta}\right\}  ,
\]

and for $\ 0<h<1,$ let $\nabla_{h}\in\mathcal{L(}H\mathcal{^{\beta})}$ the
nabla operator defined by%

\[
\left(  \nabla_{h}f\right)  (x):=\dfrac{f(x)-f(x-h)}{h},\qquad\text{for }%
x\in\left[  h,1\right]
\]

We set $\left(  \nabla_{h}f\right)  (x)=\dfrac{f(x)}{h}$ for all $0<x<h.$

Lastly, we introduce the extended nabla operator as the polygonal line with vertices
$(t_{k},\nabla f(t_{k})),$ $k=0,1,\ldots,n,$ in the following definition.

\begin{definition}
We define the operator $A_{h}$ for any $f\in H^{\beta}$ by%

\[
A_{h}f(x)=\left(  \mathcal{I}_{h}\nabla_{h}f\right)  (x)
\]

\end{definition}

Obviously $A_{h}$ is a linear bounded operator with $\left\Vert A_{h}%
\right\Vert _{\beta}\leq\dfrac{2}{h}.$

In the next proposition, it can be pointed out that the sequence $\left(  A_{h}\right)
_{h}$ has no uniform limit as $h$ tends to $0$.

\begin{proposition}
The sequence $\left(  A_{h}\right)  _{h}$ of extended nabla operator is not a
convergent sequence in $\mathcal{L}(H^{\beta})  $ as $h$ tends to $0.$
\end{proposition}

\begin{proof}
We need only to proof that $\left(  A_{h}\right)  _{h}$ is not a Cauchy sequence.

Let $\Delta_{h}$ $,\Delta_{h/2}$ be two subdivisions of $\left[  0,1\right]  $
and $f(x)=x^{\beta}.$ We then have for $t=h/2$%

\begin{align*}
\left\vert \left(  A_{h}-A_{h/2}\right)  f\left(  t\right)  \right\vert  &
=\left\vert \frac{h/2}{h}\frac{f(h)}{h}-\frac{f(h/2)}{h/2}\right\vert \\
& =\left\vert \frac{h^{\beta-1}}{2}-\left(  \frac{h}{2}\right)  ^{\beta
-1}\right\vert \geq\left\vert \frac{1}{2}-\left(  \frac{1}{2}\right)
^{\beta-1}\right\vert
\end{align*}

Since%
\[
\left\Vert \left(  A_{h}-A_{2h}\right)  f\right\Vert _{\beta}\geq\left\vert
\left(  A_{h}-A_{2h}\right)  f\left(  h/2\right)  \right\vert
\]

then%

\[\left\Vert A_{h}-A_{2h}\right\Vert _{\beta}\left\Vert f\right\Vert _{\beta
}\geq
\left\Vert \left(  A_{h}-A_{2h}\right)  f\right\Vert _{\beta}\geq\left(
\frac{1}{2}\right)  ^{\beta-1}-\frac{1}{2}%
\]

We have thus seen that $\left(  A_{h}\right)  _{h}$ is not a convergent
sequence in $\mathcal{L}(H^{\beta})  $.
\end{proof}

In the next proposition the strong convergence of extended nabla operator to
the derivative operator is proved.

\begin{proposition}
For every $f$ $\in D(A)$, such that $f^{\prime}\in H_{0}^{\beta}$ \ the
sequence $(A_{h})_{h}$ converges strongly to $A$ as $h$ tends to $0.$
\end{proposition}

\begin{proof}
Note that \
\[
\left(  A-A_{h}\right)  f(x)=\left(  I-\mathcal{I}_{h}\right)
Af(x)+\mathcal{I}_{h}\left(  A-\nabla_{h}\right)  f(x)=\left(  r_{h}f^{\prime
}\right)  (x)+\mathcal{I}_{h}\left(  A-\nabla_{h}\right)  f(x).
\]
Then%

\[
\dfrac{\left\vert \left(  A-A_{h}\right)  f(x)-\left(  A-A_{h}\right)
f(y)\right\vert }{\left\vert x-y\right\vert ^{\beta}}\leq\left\Vert \left(
r_{h}f^{\prime}\right)  \right\Vert _{\beta}+\dfrac{\left\vert \mathcal{I}%
_{h}\left(  A-\nabla_{h}\right)  f(x)-\mathcal{I}_{h}\left(  A-\nabla
_{h}\right)  f(y)\right\vert }{\left\vert x-y\right\vert ^{\beta}}%
\]
From lemma \ref{erreurinterpolation}%
\[
\left\Vert \left(  r_{h}f^{\prime}\right)  \right\Vert _{\beta}\leq
4\omega_{\beta}\left(  f^{\prime},h\right)
\]

For some $i,j , \xi_{i}\in\left]  t_{i-1},t_{i}\right[  $ and $\xi_{j}\in\left]
t_{j-1},t_{j}\right[  $ by lemma \ref{rac} we have%

\begin{align*}
\dfrac{\left\vert \mathcal{I}_{h}\left(  A-\nabla_{h}\right)  f(x)-\mathcal{I}%
_{h}\left(  A-\nabla_{h}\right)  f(y)\right\vert }{\left\vert x-y\right\vert
^{\beta}}  & \leq\dfrac{\left\vert f^{\prime}(t_{i})-\nabla_{h}f(t_{i}%
)\right\vert }{\left\vert t_{i}-t_{j}\right\vert ^{\beta}}+\dfrac{\left\vert
f^{\prime}(t_{j})-\nabla_{h}f(t_{j})\right\vert }{\left\vert t_{i}%
-t_{j}\right\vert ^{\beta}}\\
& \leq\dfrac{\left\vert f^{\prime}(t_{i})-f^{\prime}(\xi_{i})\right\vert
}{\left\vert t_{i}-t_{j}\right\vert ^{\beta}}+\dfrac{\left\vert f^{\prime
}(t_{j})-f^{\prime}(\xi_{j})\right\vert }{\left\vert t_{i}-t_{j}\right\vert
^{\beta}},\\
& \leq2\omega_{\beta}\left(  f^{\prime},h\right)
\end{align*}

Hence%

\begin{equation}
\dfrac{\left\vert \left(  A-A_{h}\right)  f(x)-\left(  A-A_{h}\right)
f(y)\right\vert }{\left\vert x-y\right\vert ^{\beta}}\leq6\omega_{\beta
}\left(  f^{\prime},h\right) \label{convergenceaalpha}%
\end{equation}

Therefore $\underset{h\rightarrow0}{\lim}\left\Vert Af-A_{h}f\right\Vert
_{\beta}=0.$
\end{proof}

\section{Fractional power of sectorial operator}

\subsection{Sectorial property}

We first, recall the Haase concept of sectorial operators \cite[Section 2.1
,p19]{haase}.
In the following $R(\lambda,B)=\left(  \lambda I-B\right)  ^{-1},\rho(B)$ and
$\sigma(B)=%
\mathbb{C}
\backslash\rho(B)$ denote respectively the resolvent, the resolvent set and
the spectrum of a linear operator $B$ on a Banach space $Z.$
Let $S_{\omega}$ denote the open sector \[\left\{  z\in\mathbb{C}
,z\neq0\text{ and }\left\vert \arg z\right\vert <\omega\right\}  ,0<\omega
\leq\pi.\]
\begin{definition}
The operator $B$ is sectorial of angle $\omega<\pi$ (in short: $B\in
Sect(\omega)$) if:

1) $\sigma(B)\subset\overline{S_{\omega}}$ \ and

2) $M(B,\omega^{\prime}):=\sup\left\{  \left\Vert \lambda R(\lambda
,B)\right\Vert ,\lambda\notin\overline{S_{\omega^{\prime}}}\right\}  <\infty$
for all $\omega<\omega^{\prime}<\pi.$

A family of operators $\left(  B_{\iota}\right)  _{\iota}$ is uniformly
sectorial of angle $\omega$ if $B_{\iota}\in Sect(\omega)$ for each $\iota$
and $\sup_{\iota}M(B_{\iota},\omega^{\prime})<\infty$ for all $\omega
<\omega^{\prime}<\pi.$
\end{definition}

\begin{remark}
Sectorial operator in Haase definition don't have to be densely defined see
\cite[Definition 3.8, p 97]{ito}.
\end{remark}

We are now able to define the fractional power of a sectorial operator
with help of the Balakrishnan representation (see \cite[Proposition
3.1.12]{haase})

\begin{proposition}
Let $B$ an operator with domain $\mathcal{D}(B),B\in Sect(\omega),$ and let
$0<\alpha<1.$ Then for all $f\in\mathcal{D}(B)$
\end{proposition}%

\begin{equation}
B^{\alpha}f(x)=-\frac{\sin\alpha\pi}{\pi}\int_{0}^{\infty}\lambda^{\alpha
-1}R(-\lambda,B)Bf(x)d\lambda\label{balakrishnan}%
\end{equation}

\subsection{Fractional power of the derivative}

We define the fractional derivative as fractional power of sectorial operator
in H\"{o}lder space. To do so, we examine sectoriality of $A.$

\begin{proposition}
The operator $A$ on $H^{\beta}$ is sectorial of angle $\dfrac{\pi}{2}.$
\end{proposition}

\begin{proof}
For all $\lambda\in\mathbb{C},$ the resolvent of the operator $A$ on $H^{\beta}$ is given by%

\[
R(\lambda,A)f(x)=-\int_{0}^{x}e^{\lambda\left(  x-t\right)  }f(t)dt
\]

Let's take $\lambda\in \mathbb{C}$ with $\operatorname{Re}(\lambda)<0$ and let $x,h$ be such that $0\leq x-h<x\leq1.$
We have
\begin{align*}
&\left\vert R(\lambda ,A)f(x)-R(\lambda ,A)f(x-h)\right\vert\\&=\left\vert
\int_{0}^{x-h}-e^{\lambda t}\left[ f\left( x-t\right) -f\left( x-h-t\right) 
\right] dt+\int_{x-h}^{x}-e^{\lambda t}f\left( x-t\right) dt\right\vert 
\end{align*}
If $p$ and $q$ are two real positive conjugates, the H\"{o}lder inequality implies:
\begin{align*}
& \left\vert \lambda\right\vert \left\vert R(\lambda,A)f(x)-R(\lambda
,A)f(x-h)\right\vert \\
& \leq\left\vert \lambda\right\vert \left(  \int_{0}^{x-h}%
e^{p\operatorname{Re}(\lambda)t}dt\right)  ^{\frac{1}{p}}\left(  \int
_{0}^{x-h}\left\vert f(x-t)-f(x-h-t)\right\vert ^{q}dt\right)  ^{\frac{1}{q}}\\%
&+\left\vert \lambda\right\vert\left(  \int_{x-h}^{x}e^{p\operatorname{Re}(\lambda)t}dt\right)  ^{\frac
{1}{p}}\left(  \int_{x-h}^{x}\left\vert f(x-t)\right\vert ^{q}\right)
^{\frac{1}{q}}dt\\
& \leq\underset{0\leq t\leq x-h}{\sup}\left\vert f(t)-f(t-h)\right\vert (x-h)^{\frac{1}{q}}
\left\vert \lambda\right\vert \left(  \frac{e^{p\operatorname{Re}%
(\lambda)\left(  x-h\right)  }-1}{p\operatorname{Re}(\lambda)}\right)
^{\frac{1}{p}}\\&+\underset{x-h\leq t\leq x}{\sup}\left\vert f(x-t)\right\vert
\left\vert (h)^{\frac{1}{q}}\lambda\right\vert \left(  \frac{e^{p\operatorname{Re}(\lambda
)x}-e^{p\operatorname{Re}(\lambda)\left(  x-h\right)  }}{p\operatorname{Re}%
(\lambda)}\right)  ^{\frac{1}{p}}\\
& \leq\omega_{\beta}\left(  f,h\right)  h^{\beta}2\left\vert \lambda
\right\vert \left(  \frac{-1}{p\operatorname{Re}(\lambda)}\right)  ^{\frac
{1}{p}}%
\end{align*}

Knowing that for every parameter $\operatorname{Re}(\lambda)<0$ the infimum, on $]1,\infty[$, of function 

\[
\left(  \frac{-1}{p\operatorname{Re}(\lambda)}\right)  ^{\frac
{1}{p}}
\]
is given by 
\[
 \underset{p>1}{\inf}\left(  \frac{-1}%
{p\operatorname{Re}(\lambda)}\right)  ^{1/p}=\left\{
\begin{array}
[c]{c}%
-\frac{1}{\operatorname{Re}(\lambda
)}  \text{ \ \ \ if }-\frac{e}{\operatorname{Re}(\lambda)}\leq1\\
 e^{\frac{1}{e}\operatorname{Re}(\lambda)}\text{
\ \ \ \ \ \  if }-\frac{e}{\operatorname{Re}(\lambda)}>1
\end{array}
\right.
\]

Consequently, by $\operatorname{Re}(\lambda)=\left\vert \lambda
\right\vert \cos\omega$ \ we have
\[
\left\vert \lambda\right\vert \left(  -\frac{1}{\operatorname{Re}(\lambda
)}\right)  =\left\vert \lambda\right\vert \left(  -\frac{1}{\left\vert
\lambda\right\vert \cos\omega}\right)  =-\frac{1}{\cos\omega}%
\]

and
\[
\left\vert \lambda\right\vert e^{\frac{1}{e}\operatorname{Re}(\lambda)}\text{
}=\left\vert \lambda\right\vert e^{\frac{1}{e}\left\vert \lambda\right\vert
\cos\omega}.
\]

In addition, the function defined on $\left[  0,\infty\right[  $ by
$xe^{x\frac{1}{e}\cos\omega}$ admits for maximum value $-\frac{1}%
{\cos\omega}.$ which gives the estimate
\[
\left\vert \lambda\right\vert \underset{p>1}{\inf}\left(  \frac{-1}%
{p\operatorname{Re}(\lambda)}\right)  ^{1/p}\leq-\frac{1}{\cos\omega}.
\]

According to the previous arguments we get
\[
\left\vert \lambda\right\vert \left\vert R(\lambda,A)f(x)-R(\lambda
,A)f(x-h)\right\vert \leq-\frac{2}{\cos\omega}\omega_{\beta}\left(
f,h\right)  h^{\beta}%
\]
which implies
\[
\left\vert \lambda\right\vert \omega_{\beta}\left(  R(\lambda,A)f,h\right)
\leq-\frac{2}{\cos\omega}\omega_{\beta}\left(  f,h\right)  h^{\beta}%
\]

and
\[
\left\vert \lambda\right\vert \left\Vert R(\lambda,A)f\right\Vert _{\beta}%
\leq-\frac{2}{\cos\omega}\left\Vert f\right\Vert _{\beta}%
\]

Therefore, for every $\lambda\in \mathbb{C}\backslash%
\overline{S_{\omega}},\dfrac{\pi}{2}<\omega\leq\pi,$%
\[
\left\Vert \lambda R(\lambda,A)\right\Vert \leq -\frac{2}{\cos\omega}
\]
\end{proof}

\begin{corollary}
Let $0<\alpha<1$ and $f\in D(A).$ Then%
\[
A^{\alpha}f(x)=\frac{1}{\Gamma(1-\alpha)}\int_{0}^{x}\left(  x-t\right)
^{-\alpha}f^{\prime}(t)dt
\]

\end{corollary}

\begin{proof}
Using the Balakrishnan representation of fractional power of sectorial
operator (\ref{balakrishnan}), the previous representation follows.
\end{proof}
In the next subsection the fractional power of operator $\nabla_{h}$ and
$A_{h}$ are constructed.

\subsection{Fractional nabla operators}

Before studying the sectoriality of $\nabla_{h}$ and $A_{h}$ we begin by a
surprising and useful result. In fact, elementary calculations show that the
operator $\mathcal{I}_{h}$ commutes with $\nabla_{h}$. This property has an
interesting consequence for the resolvent operator given in the next lemma and
stated in general framework.

\begin{lemma}
Let $\mathcal{X}$ be a Banach space, $B,T\in\mathcal{L(X)}$ such that $T$ is
idempotent and $T$ commute with $B$ then, for every $\lambda\in\rho
(B),\lambda\neq0$%

\[
R(\lambda,TB)=TR(\lambda,B)+\frac{1}{\lambda}\left(  I-T\right)
\]
\label{resTA}
\end{lemma}

\begin{proof}
To obtain the resolvent operator for $TB$ we consider the equation, for
$f,g\in\mathcal{X},$%

\[
f=\left(  \lambda I-TB\right)  g
\]

then by idempotence of the operator $T$ and commutative property we get
\[
Tf=\left(  \lambda I-B\right)  Tg
\]

combining the above two equations we have%

\[
f-Tf=\lambda\left(  I-T\right)  g
\]

using the fact that
\[
Tg=\left(  \lambda I-B\right)  ^{-1}Tf
\]

then%
\[
g=R(\lambda,B)Tf+\frac{1}{\lambda}\left(  f-Tf\right)
\]

\end{proof}

\begin{proposition}
The family $\left(  \nabla_{h}\right)  _{h}$ is uniformly sectorial of angle
$\dfrac{\pi}{2}$ on $H^{\beta}.\label{sectorialnabla}$
\end{proposition}

\begin{proof}
It can be easily proved using Laplace and inverse Laplace transforms that%

\[
R(\lambda,\nabla_{h})f(x)=-h\sum_{j=0}^{n}\frac{1}{\left(  1-\lambda h\right)
^{j+1}}f(x-t_{j})
\]

First, we check the boundedness of $R(\lambda,\nabla_{h})$ in $\mathcal{L}%
\left(  H^{\beta}\right)  .$ For every $0\leq x<y\leq1,$%

\begin{align*}
& \frac{\left\vert R(\lambda,\nabla_{h})f(x)-R(\lambda,\nabla_{h}%
)f(y)\right\vert }{\left\vert x-y\right\vert ^{\beta}}\\
& \leq\dfrac{1}{\left\vert x-y\right\vert ^{\beta}}\left\{  \left\vert
-h\sum_{j=0}^{\left[  x/h\right]  }\frac{1}{\left(  1-\lambda h\right)
^{j+1}}\left[  f(x-t_{j})-f(y-t_{j})\right]  \right\vert +\left\vert
-h\sum_{\left[  x/h\right]  +1}^{\left[  y/h\right]  }\frac{1}{\left(
1-\lambda h\right)  ^{j+1}}f(x-t_{j})\right\vert \right\} \\
& \leq\left\Vert f\right\Vert _{\beta}\sum_{j=1}^{n}\frac{h}{\left\vert
1-\lambda h\right\vert ^{j}}%
\end{align*}

Using the sum of a geometric series we have%

\begin{equation}
\frac{\left\vert R(\lambda,\nabla_{h})f(x)-R(\lambda,\nabla_{h}%
)f(y)\right\vert }{\left\vert x-y\right\vert ^{\beta}}\leq\frac{h}{\left\vert
1-\lambda h\right\vert -1}\left\Vert f\right\Vert _{\beta}\label{majresnabla}%
\end{equation}

Now,observe that for any $\lambda\in\mathbb{C}\backslash\overline{S}_{\omega}$
,$\dfrac{\pi}{2}<\omega<\pi$ we have $\left\vert \lambda h-1\right\vert >1$
and%
\[
\frac{\left\vert \lambda\right\vert \left\vert R(\lambda,\nabla_{h}%
)f(x)-R(\lambda,\nabla_{h})f(y)\right\vert }{\left\vert x-y\right\vert
^{\beta}}\leq\frac{\left\vert \lambda\right\vert h}{\left\vert 1-\lambda
h\right\vert -1}\left\Vert f\right\Vert _{\beta}%
\]
Knowing that%
\[
\left\vert 1-\lambda h\right\vert ^{2}=h^{2}\left\vert \lambda\right\vert
^{2}+1-2\left\vert \lambda\right\vert \cos\left(  \arg\lambda\right)  \geq
h^{2}\left\vert \lambda\right\vert ^{2}+1-2h\left\vert \lambda\right\vert
\cos\omega
\]
Then%
\[
\frac{\left\vert \lambda\right\vert h}{\left\vert 1-\lambda h\right\vert
-1}\leq\frac{\left\vert \lambda\right\vert h}{\sqrt{h^{2}\left\vert
\lambda\right\vert ^{2}+1-2h\left\vert \lambda\right\vert \cos\omega}-1}%
\]
Put $\varphi(z)=\dfrac{z}{\sqrt{z^{2}+1-2z\cos\omega}-1}$ for $z>0.$It is easy
to see that $\varphi(+\infty)=1,$ $\varphi(0^{+})=\dfrac{-1}{\cos\omega}$ and
the derivative satisfies%
\begin{align*}
& \varphi^{\prime}(z)\\
& =\dfrac{\left(  -1+\cos^{2}\omega\right)  z^{2}}{\left(  \sqrt
{z^{2}+1-2z\cos\omega}-1\right)  ^{2}\sqrt{z^{2}+1-2z\cos\omega}\left(
1-z\cos\omega+\sqrt{z^{2}+1-2z\cos\omega}\right)  }\\
& <0
\end{align*}
Consequently, for every $z\in\left]  0,+\infty\right[  \ 1<\varphi
(z)\leq\dfrac{-1}{\cos\omega},$ which implies that%
\begin{equation}
\frac{\left\vert \lambda\right\vert \left\vert R(\lambda,\nabla_{h}%
)f(x)-R(\lambda,\nabla_{h})f(y)\right\vert }{\left\vert x-y\right\vert
^{\beta}}\leq\dfrac{-1}{\cos\omega}\left\Vert f\right\Vert _{\beta
}\label{majreslamnabla}%
\end{equation}
We conclude that the family $\left(  \nabla_{h}\right)  _{h}$ is uniformly
sectorial of angle $\dfrac{\pi}{2}$.
\end{proof}

Consequently the extended nabla operator $A_{h}$ is also sectorial as shown in
the next corollary.

\begin{corollary}
The family $\left(  A_{h}\right)  _{h}$ is uniformly sectorial of angle
$\dfrac{\pi}{2}$ on $H^{\beta}.$
\end{corollary}

\begin{proof}
From lemma \ref{resTA} we have%

\[
R(\lambda,A_{h})f(x)=\mathcal{I}_{h}\left(  R(\lambda,\nabla_{h})f\right)
(x)+\frac{1}{\lambda}\left(  f-\mathcal{I}_{h}f\right)  (x)
\]

Then%

\begin{align*}
& \frac{\left\vert \lambda\right\vert \left\vert R(\lambda,A_{h}%
)f(x)-R(\lambda,A_{h})f(y)\right\vert }{\left\vert x-y\right\vert ^{\beta}}\\
& \leq\frac{\left\vert \lambda\right\vert \left\vert \mathcal{I}_{h}%
R(\lambda,A_{h})f(x)-\mathcal{I}_{h}R(\lambda,A_{h})f(y)\right\vert
}{\left\vert x-y\right\vert ^{\beta}}+\frac{\left\vert \left(  I-\mathcal{I}%
_{h}\right)  f(x)-\left(  I-\mathcal{I}_{h}\right)  f(y)\right\vert
}{\left\vert x-y\right\vert ^{\beta}}%
\end{align*}

from lemmas \ref{rac} and \ref{erreurinterpolation} , there exist $0\leq
m,l\leq n$ such that%

\[
\frac{\left\vert \lambda\right\vert \left\vert R(\lambda,A_{h})f(x)-R(\lambda
,A_{h})f(y)\right\vert }{\left\vert x-y\right\vert ^{\beta}}\leq
\frac{\left\vert \lambda\right\vert \left\vert R(\lambda,A_{h})f(t_{m}%
)-R(\lambda,A_{h})f(t_{l})\right\vert }{\left\vert t_{m}-t_{l}\right\vert ^{\beta}%
}+4\omega_{\beta}(f,h)
\]

From proposition \ref{sectorialnabla} we get the estimate%

\[
\frac{\left\vert \lambda\right\vert \left\vert R(\lambda,A_{h})f(x)-R(\lambda
,A_{h})f(y)\right\vert }{\left\vert x-y\right\vert ^{\beta}}\leq\left(
\dfrac{-1}{\cos\omega}+4\right)  \left\Vert f\right\Vert _{\beta}%
\]

\end{proof}

As a result we are able to define the fractional power of $\nabla_{h}$ and
$A_{h}.$ This is the purpose of the following theorem.

\begin{theorem}
Let $0<\alpha<1,$ then fractional nabla operator is%

\[
\nabla_{h}^{\alpha}f(x)=\frac{h^{1-\alpha}}{\Gamma(1-\alpha)}\overset{\left[
x/h\right]  }{\underset{i=1}{\sum}}\frac{\Gamma(j+1-\alpha)}{\Gamma
(j+1)}\nabla_{h}f(x-t_{j})
\]

and the fractional operator $A_{h}^{\alpha}$ is%

\begin{align*}
& A_{h}^{\alpha}f(x)\\
& =\sum_{k=0}^{n}\left(  \frac{x-t_{k-1}}{h}\overset{k}{\underset{i=1}{\sum}%
}\frac{\Gamma(j+1-\alpha)}{\Gamma(j+1)}\nabla_{h}f(t_{k}-t_{j})+\frac{t_{k}%
-x}{h}\overset{k-1}{\underset{i=1}{\sum}}\frac{\Gamma(j+1-\alpha)}%
{\Gamma(j+1)}\nabla_{h}f(t_{k-1}-t_{j})\right)  \mathbb{1}_{\left[  t_{k-1},t_{k}
\right]  }(x)%
\end{align*}

\end{theorem}
We call $A_{h}^{\alpha}$ the fractional extended nabla operator.
\begin{proof}
Using Balakrishnan representation of fractional power of sectorial operator
\ (\ref{balakrishnan}) , we get when $0<\alpha<1$%

\begin{align*}
\nabla_{h}^{\alpha}f(x)  & =\frac{\sin\alpha\pi}{\pi}\int_{0}^{+\infty}%
\lambda^{\alpha-1}\left(  \lambda+\nabla_{h}\right)  ^{-1}\nabla
_{h}f(x)d\lambda\\
& =-\frac{\sin\alpha\pi}{\pi}\int_{0}^{+\infty}\lambda^{\alpha-1}R\left(
-\lambda,\nabla_{h}\right)  \nabla_{h}f(x)d\lambda
\end{align*}

Then
\begin{align*}
&  \nabla_{h}^{\alpha}f(x)\\
&  =\frac{\sin\alpha\pi}{\pi}\int_{0}^{+\infty}\lambda^{\alpha-1}h\sum
_{j=0}^{\left[  x/h\right]  }\frac{1}{\left(  1+\lambda h\right)  ^{j+1}%
}\nabla_{h}f(x-t_{j})d\lambda\\
&  =h\sum_{j=0}^{\left[  x/h\right]  }\left(  \frac{\sin\alpha\pi}{\pi}\int_{0}^{+\infty}%
\lambda^{\alpha-1}\frac{1}{\left(  1+\lambda h\right)  ^{j+1}}d\lambda\right)
\nabla_{h}f(x-t_{j})
\end{align*}

Similar calculations to those \ in \cite[Theorem 3.1]{ash} give%

\[
\int_{0}^{+\infty}\lambda^{\alpha-1}\frac{1}{\left(  1+\lambda h\right)
^{j+1}}d\lambda=h^{-\alpha}\frac{\Gamma(j+1-\alpha)\Gamma(\alpha)}%
{\Gamma(j+1)}%
\]

Therfore%
\[
\nabla_{h}^{\alpha}f(x)=\frac{h^{1-\alpha}}{\Gamma(1-\alpha)}\sum
_{j=0}^{\left[  x/h\right]  }\frac{\Gamma(j+1-\alpha)}{\Gamma(j+1)}\nabla
_{h}f(x-t_{j})
\]

We now turn to the evaluation of $A_{h}^{\alpha}f$. From lemma \ref{resTA} we get%

\begin{align*}
A_{h}^{\alpha}f(x)  & =-\frac{\sin\alpha\pi}{\pi}\int_{0}^{+\infty}%
\lambda^{\alpha-1}R\left(  -\lambda,A_{h}\right)  A_{h}f(x)d\lambda\\
& =-\frac{\sin\alpha\pi}{\pi}\int_{0}^{+\infty}\lambda^{\alpha-1}%
\mathcal{I}_{h}R\left(  -\lambda,\nabla_{h}\right)  \mathcal{I}_{h}\nabla
_{h}f(x)d\lambda\\
& =-\frac{\sin\alpha\pi}{\pi}\int_{0}^{+\infty}\lambda^{\alpha-1}%
\mathcal{I}_{h}R\left(  -\lambda,\nabla_{h}\right)  \nabla_{h}f(x)d\lambda
\end{align*}

the required evaluation of $A_{h}^{\alpha}f$ \ then follows.
\end{proof}

\begin{remark}
The operator $\nabla_{h}^{\alpha}$ is nothing but the Gr\"{u}nwald-Letnikov operator%

\[
\nabla_{h}^{\alpha}f(x)=h^{-\alpha}\overset{\left[  x/h\right]  }%
{\underset{j=0}{\sum}}(-1)^{j}\binom{\alpha}{j}f(x-jh)
\]

Let us mention that this operator was defined in a formal way as a
generalization of difference formulas of integer order by replacing the
integer order by a real number.

It is worth to underscore that in \cite{ash}\ discrete Gr\"{u}nwald-Letnikov
approximations was called the Riemann-Liouville fractional derivative of order
$\alpha.$
\end{remark}

The remaining problem is to study if the strong convergence $(A_{h}f)_{h}$ to
$Af$ can gives rise to the convergence of power operators; this is the aim of
the next section.

\section{H\"{o}lderian convergence of fractional extended nabla operator to fractional derivative}

The following  result, known elsewhere, is given in a suitable form for later uses.
\begin{lemma}
There exists a function $\Phi_{\alpha},$ such that%
\begin{center}
$\forall m\geq1\qquad\dfrac{\Gamma(m+\alpha)}{\Gamma(m+1)}=m^{\alpha-1}+\dfrac{%
1}{\Gamma(1-\alpha)}\Phi_{\alpha}(m)$,\\ with\; $\left\vert \Phi_{\alpha}(m)\right\vert $ $\leq \dfrac{
\Gamma(2-\alpha)}{2}m^{\alpha-2}$.
\end{center}
\label{rapportdegamma}
\end{lemma}

\begin{proof}
From the definition of the beta function we have%
\begin{equation}
\dfrac{\Gamma(m+\alpha)}{\Gamma(m+1)}=\frac{1}{\Gamma(1-\alpha)}\int_{0}%
^{1}t^{(m+\alpha-1)}(1-t)^{-\alpha}dt \label{beta}%
\end{equation}
Let $t=e^{-u}$ then the equality (\ref{beta}) becomes%
$$
\dfrac{\Gamma(m+\alpha)}{\Gamma(m+1)} =\frac{1}{\Gamma(1-\alpha)}\int
_{0}^{+\infty}e^{-mu}(e^{u}-1)^{-\alpha}du =\frac{1}{\Gamma(1-\alpha)}\int_{0}^{+\infty}e^{-mu}u^{-\alpha}(\dfrac
{u}{e^{u}-1})^{\alpha}du
$$
Using the generating function of the Bernoulli numbers%
\[
G(u)=\dfrac{u}{e^{u}-1}=\overset{+\infty}{\underset{k=0}{\sum}}B_{k}%
\dfrac{u^{k}}{k!}=1+\varphi(u)>0
\]
where $\varphi:\left[  0,+\infty\right[  \rightarrow\left] -1,0\right]  $ is a
continuous function. We have%
\[
\dfrac{\Gamma(m+\alpha)}{\Gamma(m+1)}=\frac{1}{\Gamma(1-\alpha)}\int
_{0}^{+\infty}e^{-mu}u^{-\alpha}(1+\varphi(u))^{\alpha}du
\]
Now, Taylor's formula with integral remainder applied to the function
$(1+\varphi(u))^{\alpha}$ gives $\ $%
\[
(1+\varphi(u))^{\alpha}=1+\alpha\varphi(u)\int_{0}^{1}(1+\xi\varphi
(u))^{\alpha-1}d\xi
\]
Therefore%
$$
\dfrac{\Gamma(m+\alpha)}{\Gamma(m+1)}  =\frac{1}{\Gamma(1-\alpha)}\int
_{0}^{+\infty}e^{-mu}u^{-\alpha}du+\frac{\alpha}{\Gamma(1-\alpha)}\int_{0}^{+\infty}e^{-mu}u^{-\alpha}%
\varphi(u)\int_{0}^{1}(1+\xi\varphi(u))^{\alpha-1}d\xi du
$$
and then%
\begin{equation*}
\dfrac{\Gamma(m+\alpha)}{\Gamma(m+1)}=m^{\alpha-1}+\frac{1}%
{\Gamma(1-\alpha)}\Phi_{\alpha}(m)
\end{equation*}
where
\[
\Phi_{\alpha}(m)=\alpha\int_{0}^{+\infty}e^{-mu}u^{-\alpha}\varphi(u)\int_{0}^{1}%
(1+\xi\varphi(u))^{\alpha-1}d\xi du
\]
From the identity\;$1+\xi\varphi(u)=1-\xi+\xi(1+ \varphi(u))$\;
and the fact that\; $1+\varphi(u)>0$\; for every\; $u\geq0$,  we have
\[
1+\xi\varphi(u)\geq1-\xi\qquad\text{and}\qquad\ \left(  1+\xi\varphi
(u)\right)  ^{\alpha-1}\leq\left(  1-\xi\right)  ^{\alpha-1}%
\]
Consequently%
\[
\int_{0}^{1}(1+\xi\varphi(u))^{\alpha-1}d\xi\leq\int_{0}^{1}(1-\xi)^{\alpha
-1}d\xi\leq\dfrac{1}{\alpha}
\]
Hence%
\begin{equation}
\alpha\left\vert \int_{0}^{+\infty}e^{-mu}u^{-\alpha}\varphi(u)\int_{0}^{1}%
(1+\xi\varphi(u))^{\alpha-1}d\xi du\right\vert \leq\int_{0}^{+\infty}%
e^{-mu}u^{-\alpha}\left\vert \varphi(u)\right\vert du \label{ineq3}
\end{equation}
The function $\dfrac{\varphi(u)}{u}$ is strictly increasing on $\left[
0,+\infty\right]  $, $\underset{u\rightarrow 0^{+}}{\lim }\dfrac{\varphi(u)}{u}=-\dfrac
{1}{2}$ and $\underset{u\rightarrow \infty}{\lim }\dfrac{\varphi(u)}{u}=0$ .\\
So, $\left\vert \dfrac{\varphi(u)}{u}\right\vert \leq\dfrac{1}{2}$, and the
inequality (\ref{ineq3}) becomes%
\[
\left\vert \Phi_{\alpha}(m)\right\vert \leq\dfrac{1}{2}\int_{0}^{+\infty}e^{-mu}%
u^{1-\alpha}du
\]
and%
\begin{equation*}
\left\vert \Phi_{\alpha}(m)\right\vert \leq\dfrac{\Gamma(2-\alpha)m^{\alpha-2}}{2}
\end{equation*}
\end{proof}

Before stating the convergence theorem, we define for all $f$ in $D(A),$ the function $\varphi$  by
\[
\varphi(x):=\frac{1}{\Gamma\left(  1-\alpha\right)  }\int_{0}^{x}%
(x-t)^{-\alpha}A_{h}f\left(  t\right)  dt
\]

For the construction of the convergence result proof we need the following lemmas

\begin{lemma}

For all $0<\beta<1$ such that $1-\alpha-\beta>0$ we have

\[
\omega_{\beta}(\varphi,h)\leq\frac{8}{\Gamma\left(  2-\alpha\right)
}\left\Vert f^{\prime}\right\Vert _{\beta}h^{1-\alpha-\beta}%
\]
\label{omegaphi}
\end{lemma}

\begin{proof}
Remark that the function $\varphi$ satisfies the following estimation%

\[
\frac{\left\vert \varphi(x)-\varphi(y)\right\vert }{\left\vert x-y\right\vert
^{\beta}}\leq\frac{1}{\Gamma\left(  2-\alpha\right)  }\left\Vert
A_{h}f\right\Vert _{\beta}\left(  \max(x,y)\right)  ^{1-\alpha}\leq\frac
{1}{\Gamma\left(  2-\alpha\right)  }\left\Vert A_{h}f\right\Vert _{\beta}%
\]

then $\varphi\in H^{\beta}\left[  0,1\right]  .$

Let us now estimate $\omega_{\beta}(\varphi,h),$ we distinguish two cases,

First $t_{k-1}<x<y\leq t_{k}:$

Notice that%
\begin{align*}
\Gamma\left(  1-\alpha\right)  \varphi(x)  & =\overset{k-1}{\underset
{i=1}{\sum}}\int_{t_{i-1}}^{t_{i}}(x-t)^{-\alpha}\left(  \frac{t-t_{i-1}}%
{h}\nabla_{h}f(t_{i})+\frac{t_{i}-t}{h}\nabla_{h}f(t_{i-1})\right)  dt+\\
& \int_{t_{k-1}}^{x}(x-t)^{-\alpha}\left(  \frac{t-t_{k-1}}{h}\nabla
_{h}f(t_{k})+\frac{t_{k}-t}{h}\nabla_{h}f(t_{k-1})\right)  dt
\end{align*}

Hence%

\begin{align*}
&  \Gamma\left(  1-\alpha\right)  \left(  \varphi(x)-\varphi(y)\right) \\
&  =\overset{k-1}{\underset{i=1}{\sum}}\int_{t_{i-1}}^{t_{i}}\left(
(x-t)^{-\alpha}-(y-t)^{-\alpha}\right)  \left(  \frac{t-t_{i-1}}{h}\nabla
_{h}f(t_{i})+\frac{t_{i}-t}{h}\nabla_{h}f(t_{i-1})\right)  dt\\
&  +\int_{t_{k-1}}^{x}\left(  (x-t)^{-\alpha}-(y-t)^{-\alpha}\right)  \left(
\frac{t-t_{k-1}}{h}\nabla_{h}f(t_{k})+\frac{t_{k}-t}{h}\nabla_{h}%
f(t_{k-1})\right)  dt+\\
&  \int_{x}^{y}(y-t)^{-\alpha}\left(  \frac{t-t_{k-1}}{h}\nabla_{h}%
f(t_{k})+\frac{t_{k}-t}{h}\nabla_{h}f(t_{k-1})\right)  dt
\end{align*}

and%

\begin{align*}
&\Gamma\left(  1-\alpha\right)     \left\vert \varphi(x)-\varphi(y)\right\vert
\\
&  \leq2\left\Vert f^{\prime}\right\Vert _{\beta}\left\{  \overset
{k-1}{\underset{i=1}{\sum}}\int_{t_{i-1}}^{t_{i}}\left(  (x-t)^{-\alpha
}-(y-t)^{-\alpha}\right)  dt+\int_{t_{k-1}}^{x}\left(  (x-t)^{-\alpha
}-(y-t)^{-\alpha}\right)  dt+\int_{x}^{y}(y-t)^{-\alpha}\right\}
\end{align*}

\bigskip which leads to%

\begin{align*}
& \left\vert \varphi(x)-\varphi(y)\right\vert \\
& \leq\frac{2}{\Gamma\left(  2-\alpha\right)  }\left\Vert f^{\prime
}\right\Vert _{\beta}\overset{k-1}{\underset{i=1}{\sum}}\left(  (x-t_{i-1}%
)^{1-\alpha}-(y-t_{i-1})^{1-\alpha}-(x-t_{i})^{1-\alpha}+(y-t_{i})^{1-\alpha
}\right)  \\
& +\frac{2}{\Gamma\left(  2-\alpha\right)  }\left\Vert f^{\prime}\right\Vert
_{\beta}\left[  (x-t_{k-1})^{1-\alpha}-(y-t_{k-1})^{1-\alpha}+2(y-x)^{1-\alpha
}\right]
\end{align*}

finally%
\begin{align*}
\frac{\left\vert \varphi(x)-\varphi(y)\right\vert }{\left\vert x-y\right\vert
^{\beta}}  & \leq\frac{4}{\Gamma\left(  2-\alpha\right)  }\left\Vert
f^{\prime}\right\Vert _{\beta}\left\vert y-x\right\vert ^{1-\alpha-\beta}\\
& \leq\frac{4}{\Gamma\left(  2-\alpha\right)  }\left\Vert f^{\prime
}\right\Vert _{\beta}h^{1-\alpha-\beta}%
\end{align*}

Second $t_{k-1}<x\leq t_{k}<y\leq t_{k+1}:$%
\begin{gather*}
\frac{\left\vert \varphi(x)-\varphi(y)\right\vert }{\left\vert x-y\right\vert
^{\beta}}\leq\frac{\left\vert \varphi(x)-\varphi(t_{k})\right\vert
}{\left\vert x-t_{k}\right\vert ^{\beta}}+\frac{\left\vert \varphi
(t_{k})-\varphi(y)\right\vert }{\left\vert t_{k}-y\right\vert ^{\beta}}\\
\leq\frac{8}{\Gamma\left(  2-\alpha\right)  }\left\Vert f^{\prime}\right\Vert
_{\beta}h^{1-\alpha-\beta}%
\end{gather*}

Therefore%
\[
\omega_{\beta}(\varphi,h)\leq\frac{8}{\Gamma\left(  2-\alpha\right)
}\left\Vert f^{\prime}\right\Vert _{\beta}h^{1-\alpha-\beta}%
\]

\end{proof}

\begin{lemma}
There exists $C>0$ such that  for every $0\leq k$ $\leq n$ 

\label{majphinabla}%
\[
\frac{\left\vert \varphi(t_{k})-\nabla_{h}^{\alpha}f(t_{k})\right\vert
}{h^{\beta}}\leq\left(  1+\frac{C}{\Gamma\left(  1-\alpha\right)  }\right)
\left\Vert f^{\prime}\right\Vert _{\beta}h^{1-\alpha-\beta}+\frac{2^{\beta}%
}{\Gamma\left(  2-\alpha\right)  }\omega_{\beta}(f^{\prime},2h)
\]

\end{lemma}

\begin{proof}
Obviously if $k=0$ the lemma holds for every $C>0.$ Assume now that $k>0,$
then%
\[
\varphi(t_{k})=\frac{1}{\Gamma\left(  1-\alpha\right)  }\overset{k}%
{\underset{i=1}{\sum}}\int_{t_{j-1}}^{t_{j}}(t_{k}-t)^{-\alpha}\left[
\frac{t-t_{j-1}}{h}\nabla_{h}f\left(  t_{j}\right)  +\frac{t_{j}-t}{h}%
\nabla_{h}f\left(  t_{j-1}\right)  \right]  dt
\]

A simple integration leads to%

\begin{align*}
& \varphi(t_{k})\\
& =\frac{1}{\Gamma\left(  2-\alpha\right)  }\overset{k}{\underset{i=1}{\sum}%
}(t_{k}-t_{j-1})^{1-\alpha}\nabla_{h}f\left(  t_{j-1}\right)  -(t_{k}%
-t_{j})^{1-\alpha}\nabla_{h}f\left(  t_{j}\right)  \\
& +\frac{1}{\Gamma\left(  2-\alpha\right)  }\overset{k}{\underset{i=1}{\sum}%
}\frac{(t_{k}-t_{j-1})^{2-\alpha}-(t_{k}-t_{j})^{2-\alpha}}{(2-\alpha
)h}\left[  \nabla_{h}f\left(  t_{j}\right)  -\nabla_{h}f\left(  t_{j-1}%
\right)  \right]
\end{align*}

which can be arranged as follows%

\begin{align*}
& \varphi(t_{k})\\
& =\frac{1}{\Gamma\left(  2-\alpha\right)  }\overset{k-1}{\underset{i=1}{\sum
}}\left[  (t_{k}-t_{j})^{1-\alpha}-(t_{k}-t_{j+1})^{1-\alpha}\right]
\nabla_{h}f\left(  t_{j}\right)  \\
& +\frac{1}{\Gamma\left(  2-\alpha\right)  }\overset{k}{\underset{i=1}{\sum}%
}\left(  \frac{(t_{k}-t_{j-1})^{2-\alpha}-(t_{k}-t_{j})^{2-\alpha}}%
{(2-\alpha)h}-(t_{k}-t_{j})^{1-\alpha}\right)  \left(  \nabla_{h}f\left(
t_{j}\right)  -\nabla_{h}f\left(  t_{j-1}\right)  \right)
\end{align*}

Then %

\[
\varphi(t_{k})-\nabla_{h}^{\alpha}f(t_{k})=S_{1}+S_{2}-h^{1-\alpha}\nabla_{h}%
f(t_{k})
\]

where %

\[
S_{1}=\frac{h^{1-\alpha}}{\Gamma\left(  1-\alpha\right)  }\overset{k-1}%
{\underset{i=1}{\sum}}\left[  \frac{j^{1-\alpha}-(j-1)^{1-\alpha}}{1-\alpha
}-\frac{\Gamma(j+1-\alpha)}{\Gamma(j+1)}\right]  \nabla_{h}f\left(
t_{k}-t_{j}\right)
\]

and %

\[
S_{2}=\frac{1}{\Gamma\left(  2-\alpha\right)  }\overset{k}{\underset{i=1}{\sum}%
}\left(  \frac{(t_{k}-t_{j-1})^{2-\alpha}-(t_{k}-t_{j})^{2-\alpha}}%
{(2-\alpha)h}-(t_{k}-t_{j})^{1-\alpha}\right)  \left(  \nabla_{h}f\left(
t_{j}\right)  -\nabla_{h}f\left(  t_{j-1}\right)  \right)
\]

By using the fact that %

\begin{gather*}
0\leq\frac{(t_{k}-t_{j-1})^{2-\alpha}-(t_{k}-t_{j})^{2-\alpha}}{(2-\alpha
)h}-(t_{k}-t_{j})^{1-\alpha}=\frac{1}{h}\int_{t_{j-1}}^{t_{j}}\left(
(t_{k}-t)^{1-\alpha}-(t_{k}-t_{j})^{1-\alpha}\right)  dt\\
\leq(t_{k}-t_{j-1})^{1-\alpha}-(t_{k}-t_{j})^{1-\alpha}%
\end{gather*}

and %

\[
\left\vert \nabla_{h}f\left(  t_{j}\right)  -\nabla_{h}f\left(  t_{j-1}%
\right)  \right\vert \leq\left(  2h\right)  ^{\beta}\omega_{\beta}(f^{\prime
},2h)
\]

$\left\vert S_{2}\right\vert $ can be estimated by %

\[
\left\vert S_{2}\right\vert \leq\frac{\left(  2h\right)  ^{\beta}\omega
_{\beta}(f^{\prime},2h)}{\Gamma\left(  2-\alpha\right)  }%
\]

It remain to estimate $\left\vert S_{1}\right\vert$. To doing so, we use the
lemma \ref{rapportdegamma},%
\[
\dfrac{\Gamma(j+1-\alpha)}{\Gamma(j+1)}=j^{-\alpha}+\frac{1}%
{\Gamma(\alpha)}\Phi_{1-\alpha}\left(  j\right)
\]

with \qquad%
\[
\left\vert \Phi_{1-\alpha}\left(  j\right)  \right\vert \leq\frac{\Gamma(1+\alpha)}%
{2}j^{-\alpha-1}%
\]

Therefore%
\[
\frac{j^{1-\alpha}-(j-1)^{1-\alpha}}{1-\alpha}-\frac{\Gamma(j+1-\alpha
)}{\Gamma(j+1)}=\int_{t_{j-1}}^{t_{j}}\left(  s^{-\alpha}-j^{-\alpha}\right)
ds-\frac{1}{\Gamma(\alpha)}\Phi_{1-\alpha}\left(  j\right)
\]

and for every $j\geq2$%

\[
\left\vert \frac{j^{1-\alpha}-(j-1)^{1-\alpha}}{1-\alpha}-\frac{\Gamma
(j+1-\alpha)}{\Gamma(j+1)}\right\vert \leq\left(  j-1\right)  ^{-\alpha
}-j^{-\alpha}+\frac{1}{\Gamma(\alpha)}\left\vert \Phi_{1-\alpha}\left(  j\right)
\right\vert
\]

This leads to%

\begin{align*}
& \overset{k-1}{\underset{j=1}{\sum}}\left\vert \frac{j^{1-\alpha
}-(j-1)^{1-\alpha}}{1-\alpha}-\frac{\Gamma(j+1-\alpha)}{\Gamma(j+1)}%
\right\vert \\
& \leq\frac{1}{1-\alpha}-\Gamma\left(  2-\alpha\right)  +1-\left(  k-1\right)
^{-\alpha}+\frac{\alpha  }{2}\overset{k-1}%
{\underset{j=1}{\sum}}j^{-\alpha-1}\\
& \leq C
\end{align*}

with
\[
C=\frac{1}{1-\alpha}-\Gamma\left(  2-\alpha\right)  +1+\frac{\alpha
  }{2}\zeta(1+\alpha)>0
\]

where $\zeta(\cdot)$ is the Riemann zeta function.

Finally%

\[
\left\vert S_{1}\right\vert \leq\frac{h^{1-\alpha}}{\Gamma\left(
1-\alpha\right)  }C\left\Vert f^{\prime}\right\Vert _{\beta}%
\]

Now we can put the pieces together to get%

\[
\frac{\left\vert \varphi(t_{k})-\nabla_{h}^{\alpha}f(t_{k})\right\vert
}{h^{\beta}}\leq\left(  1+\frac{C}{\Gamma\left(  1-\alpha\right)  }\right)
\left\Vert f^{\prime}\right\Vert _{\beta}h^{1-\alpha-\beta}+\frac{2^{\beta}%
}{\Gamma\left(  2-\alpha\right)  }\omega_{\beta}(f^{\prime},2h)
\]

\end{proof}

The following theorem shows that the sequence $\left(  A_{h}^{\alpha}\right)
_{h}$ converges strongly to $A^{\alpha}.$

\begin{theorem}
Let $X_{\beta}$ be$\ $the space $X_{\beta}=\left\{  f\in H^{\beta}\text{ such
that }f^{\prime}\in H_{0}^{\beta}\right\}  $. 
\\Then for all $\beta$ such that $1-\alpha-\beta>0$ the sequence $\left(
A_{h}^{\alpha}\right)  _{h}$ converges strongly to the fractional derivative
$A^{\alpha}$ on $X_{\beta}$ \ as $h$ tends to $0.$
\end{theorem}

\begin{proof}
For every $0\leq x<y\leq1,$%
\begin{align*}
&  \left(  A^{\alpha}-A_{h}^{\alpha}\right)  (f)(x)-\left(  A^{\alpha}%
-A_{h}^{\alpha}\right)  (f)(y)\\
&  =-\frac{\sin\pi\alpha}{\pi}\int_{0}^{+\infty}\lambda^{\alpha-1}\left(
R\left(  -\lambda,A\right)  \left(  Af\right)  -R\left(  -\lambda
,A_{h}\right)  \left(  A_{h}f\right)  \right)  \left(  x\right)  d\lambda\\
&  +\frac{\sin\pi\alpha}{\pi}\int_{0}^{+\infty}\lambda^{\alpha-1}\left(
R\left(  -\lambda,A\right)  \left(  Af\right)  -R\left(  -\lambda
,A_{h}\right)  \left(  A_{h}f\right)  \right)  \left(  y\right)  d\lambda
\end{align*}

by introducing a mixed term we get
\begin{align*}
& \left(  A^{\alpha}-A_{h}^{\alpha}\right)  (f)(x)-\left(  A^{\alpha}%
-A_{h}^{\alpha}\right)  (f)(y)\\
& =-\frac{\sin\pi\alpha}{\pi}\int_{0}^{+\infty}\lambda^{\alpha-1}R\left(
-\lambda,A\right)  \left(  \left(  Af-A_{h}f\right)  \left(  x\right)
-\left(  Af-A_{h}f\right)  \left(  y\right)  \right)  d\lambda\\
& +\frac{\sin\pi\alpha}{\pi}\int_{0}^{+\infty}\lambda^{\alpha-1}\left(
R\left(  -\lambda,A\right)  -R\left(  -\lambda,A_{h}\right)  \right)  \left(
A_{h}f\left(  x\right)  -\left(  A_{h}f\right)  \left(  y\right)  \right)
d\lambda
\end{align*}

Denote by $I_{1}(x,y)$ and $I_{2}(x,y)$ respectively the first and the second
integral in the equality above.

We begin by the first integral
\begin{align*}
& \left\vert I_{1}(x,y)\right\vert \\
& \leq\int_{0}^{+\infty}\lambda^{\alpha-1}\left\vert R\left(  -\lambda
,A\right)  \left(  \left(  Af-A_{h}f\right)  \left(  x\right)  -\left(
Af-A_{h}f\right)  \left(  y\right)  \right)  \right\vert d\lambda\\
& \leq\int_{0}^{+\infty}\lambda^{\alpha-1}\int_{0}^{x}e^{-\lambda t}\left\vert
\left(  Af-A_{h}f\right)  \left(  y-t\right)  -\left(  Af-A_{h}f\right)
\left(  x-t\right)  \right\vert dtd\lambda\\
& +\int_{0}^{+\infty}\lambda^{\alpha-1}\int_{x}^{y}e^{-\lambda t}\left\vert
\left(  Af-A_{h}f\right)  \left(  y-t\right)  \right\vert dtd\lambda
\end{align*}

which leads to
\[
\frac{I_{1}(x,y)}{\left\vert x-y\right\vert ^{\beta}}\leq\left\Vert
Af-A_{h}f\right\Vert _{\beta}\int_{0}^{+\infty}\lambda^{\alpha-1}\left(
\int_{0}^{y}e^{-\lambda t}dt\right)  d\lambda
\]

The estimate%

\[
\frac{I_{1}(x,y)}{\left\vert x-y\right\vert ^{\beta}}\leq\frac{6\Gamma
(\alpha)}{1-\alpha}\omega_{\beta}\left(  f^{\prime},h\right)
\]

follows from Fubini's theorem and inequality \ref{convergenceaalpha} .

Consider now the second integral

First notice that%
\begin{align*}
& \frac{\sin\pi\alpha}{\pi}\int_{0}^{+\infty}\lambda^{\alpha-1}R\left(
-\lambda,A\right)  A_{h}f(x)d\lambda\\
& =\frac{\sin\pi\alpha}{\pi}\int_{0}^{+\infty}\lambda^{\alpha-1}\left(
\int_{0}^{x}e^{-\lambda(x-t)}A_{h}f\left(  t\right)  dt\right)  d\lambda\\
& =\frac{\sin\pi\alpha}{\pi}\int_{0}^{x}\left(  \int_{0}^{+\infty}%
\lambda^{\alpha-1}e^{-\lambda(x-t)}d\lambda\right)  A_{h}f\left(  t\right)
dt\\
& =\dfrac{1}{\Gamma(1-\alpha)}\int_{0}^{x}(x-t)^{-\alpha}A_{h}f\left(
t\right)  dt
\end{align*}

Then
\begin{align*}
& \frac{\sin\pi\alpha}{\pi}\int_{0}^{+\infty}\lambda^{\alpha-1}\left(
R\left(  -\lambda,A\right)  -R\left(  -\lambda,A_{h}\right)  \right)
A_{h}f\left(  x\right)  d\lambda\\
& =\dfrac{1}{\Gamma(1-\alpha)}\int_{0}^{x}(x-t)^{-\alpha}A_{h}f\left(
t\right)  dt-\mathcal{I}_{h}\nabla_{h}^{\alpha}f(x)\\
& =(r_{h}\varphi)(x)+\mathcal{I}_{h}\left(  \varphi(x)-\nabla_{h}^{\alpha
}f(x)\right)
\end{align*}

From lemmas \ref{erreurinterpolation} and \ref{rac} we have for some $k$ and $m$ 
\[\frac{\sin\pi\alpha}{\pi}
\frac{\left\vert I_{2}\left(  x,y\right)  \right\vert }{\left\vert
x-y\right\vert ^{\beta}}\leq4\omega_{\beta}(\varphi,h)+\frac{\left\vert
\varphi(t_{k})-\nabla_{h}^{\alpha}f(t_{k})-\varphi(t_{m})+\nabla_{h}^{\alpha
}f(t_{m})\right\vert }{\left\vert t_{k}-t_{m}\right\vert ^{\beta}}%
\]

From lemmas \ref{omegaphi} and \ref{majphinabla} we deduce
\[\frac{\sin\pi\alpha}{\pi}
\frac{\left\vert I_{2}\left(  x,y\right)  \right\vert }{\left\vert
x-y\right\vert ^{\beta}}\leq2\left(  \frac{16}{\Gamma\left(  2-\alpha\right)  }+\frac{C}{\Gamma\left(
1-\alpha\right)  }+1\right)  \left\Vert f^{\prime}\right\Vert _{\beta
}h^{1-\alpha-\beta}+\frac{2^{\beta+1}}{\Gamma\left(  2-\alpha\right)  }%
\omega_{\beta}(f^{\prime},2h)\]
Hence
\begin{align*}
& \left\Vert \left(  A^{\alpha}-\nabla_{h}^{\alpha}\right)  (f)\right\Vert
_{\beta}\\
& \leq2\left(  \frac{16}{\Gamma\left(  2-\alpha\right)  }+\frac{C}{\Gamma\left(
1-\alpha\right)  }+1\right)  \left\Vert f^{\prime}\right\Vert _{\beta
}h^{1-\alpha-\beta}+\frac{2^{\beta+1}+6}{\Gamma\left(  2-\alpha\right)  }
\omega_{\beta}(f^{\prime},2h)\
\end{align*}

and the conclusion of the theorem holds.
\end{proof}

\section{Numerical examples}
In this section two examples are discussed.
\subsection{Exemple 1}

Consider the fractional derivative of $f(x)=x^{\mu}\ln x.$ The analytical
expression of the fractional derivative of $f$ is%

\[
A^{\alpha}f(x)=\frac{\Gamma(\mu+1)}{\Gamma(\mu+1-\alpha)}x^{\mu-\alpha}\left[
\ln x+\psi(\mu+1)-\psi(\mu+1-\alpha)\right]
\]

Where $\psi(\cdot)$ denote the digamma function see \cite[Formula (103)]{valerio}.
\\In the next tables error at the step size ${h}$ is the H\"{o}lderian error 
defined by 
\begin{equation}
\underset{0\leq i<j\leq1/h}{\max}\frac{\left\vert \left(  A^{\alpha}%
f-\nabla_{h}^{\alpha}f\right)  \left(  t_{i}\right)  -\left(  A^{\alpha
}f-\nabla_{h}^{\alpha}f\right)  \left(  t_{j}\right)  \right\vert }{\left\vert
t_{i}-t_{j}\right\vert ^{\beta}}\label{error}%
\end{equation}
\\According to our theoretical consideration the convergence is ensured by $h^{1-\alpha-\beta}$ and  $\omega_{\beta}(f^{\prime},2h)$.
\\Let us establish an estimation of $\omega_{\beta}(f^{\prime},h)$
\\For every $0\leq x<y\leq 1$
\[
f^{\prime}(y)-f^{\prime}(x)=\mu(y^{\mu-1}\ln y-x^{\mu-1}\ln x)+y^{\mu
-1}-x^{\mu-1}%
\]
Using the fact that 
\[
y^{\mu-1}\ln y-x^{\mu-1}\ln x=\int_{x}^{y}\frac{d}{dt}(t^{\mu-1}\ln
t)dt=\left(  \mu-1\right)  \int_{x}^{y}t^{\mu-2}\ln tdt+\frac{1}{\mu-1}\left(
y^{\mu-1}-x^{\mu-1}\right)
\]

Then for every $1+\beta<\beta^{\prime}<\mu$ we have%

\[
\int_{x}^{y}t^{\mu-2}\ln tdt=\int_{x}^{y}t^{\mu-\beta^{\prime}+\beta^{\prime
}-2}\ln tdt
\]

Setting $M=\underset{t\in\left]  0,1\right]  }{\max}\left\vert t^{\mu
-\beta^{\prime}}\ln t\right\vert $

we have

\[
\left\vert \int_{x}^{y}t^{\mu-2}\ln tdt\right\vert \leq M\int_{x}^{y}%
t^{\beta^{\prime}-2}dt=\frac{M}{\beta^{\prime}-1}\left(  y^{\beta^{\prime}%
-1}-x^{\beta^{\prime}-1}\right)
\]

Therefore%

\[
\left\vert f^{\prime}(y)-f^{\prime}(x)\right\vert \leq\frac{M\mu\left(
\mu-1\right)  }{\beta^{\prime}-1}\left(  y^{\beta^{\prime}-1}-x^{\beta
^{\prime}-1}\right)  +\left(  \frac{\mu}{\mu-1}+1\right)  \left(  y^{\mu
-1}-x^{\mu-1}\right)
\]

It follows that 
\begin{align*}
\frac{\left\vert f^{\prime}(y)-f^{\prime}(x)\right\vert }{\left\vert
y-x\right\vert ^{\beta}}  &  \leq\frac{M\mu\left(  \mu-1\right)  }%
{\beta^{\prime}-1}\left\vert y-x\right\vert ^{\beta^{\prime}-1-\beta}+\left(
\frac{\mu}{\mu-1}+1\right)  \left\vert y-x\right\vert ^{\mu-1-\beta}\\
&  \leq\left(  \frac{M\mu\left(  \mu-1\right)  }{\beta^{\prime}-1}+\left(
\frac{\mu}{\mu-1}+1\right)  \left\vert y-x\right\vert ^{\mu-\beta^{\prime}%
}\right)  \left\vert y-x\right\vert ^{\beta^{\prime}-1-\beta}%
\end{align*}

If $\left\vert y-x\right\vert \leq h$ then

\[
\frac{\left\vert f^{\prime}(y)-f^{\prime}(x)\right\vert }{\left\vert
y-x\right\vert ^{\beta}}\leq\left(  \frac{M\mu\left(  \mu-1\right)  }%
{\beta^{\prime}-1}+\left(  \frac{\mu}{\mu-1}+1\right)  h^{\mu-\beta^{\prime}%
}\right)  h^{\beta^{\prime}-1-\beta}%
\]
and 

\[
\omega_{\beta}(f^{\prime},h)\leq\left(  \frac{M\mu\left(  \mu-1\right)
}{\beta^{\prime}-1}+\left(  \frac{\mu}{\mu-1}+1\right)  h^{\mu-\beta^{\prime}%
}\right)  h^{\mu-1-\beta}%
\]

In case $\mu=3/2,\alpha=0.3,\beta=0.1$ the convergence is ensured since \[\mu-1-\beta=0.4>0.\]

The results concerning errors are presented in table \ref{table1} for $\mu=3/2,\\ \alpha=0.3,\beta=0.1$

\begin{table}[h]
\centering
\begin{tabular}
[c]{|l|l|l|}\hline
$h$ & Error \\\hline
$2^{-6}$ & 0.0079082 \\\hline
$2^{-7}$ & 0.0040833\\\hline
$2^{-8}$ & 0.0021392\\\hline
$2^{-9}$ & 0.0011478 \\\hline
$2^{-10}$ & 0.0006054 \\\hline
$2^{-11}$ & 0.0003150\\\hline
$2^{-12}$ & 0.0001622 \\\hline
\end{tabular}
\caption{ Error defined by (\ref{error})  when $f(x)=x^{\mu}\ln x$  for $\mu=3/2,\alpha=0.3$ and  $\beta=0.1$. }

\label{table1} 

\end{table}

In figure \ref{fig}, on the left the graphs of $A^{\alpha}f$ and $A_{h}^{\alpha}f$ are shown. On the right we give the H\"{o}lderian errors. 
\begin{figure}

     \centering
     \begin{subfigure}[b]{0.45\textwidth}
         \centering
         \includegraphics[width=\textwidth]{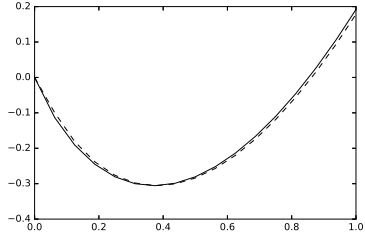}
         \caption{$A^{\alpha}f$ in the continuous line, $A_{h}^{\alpha}f$ in dotted line for $h=2^{-4}$.}
      \end{subfigure}
     \hfill
     \begin{subfigure}[b]{0.45\textwidth}
         \centering
         \includegraphics[width=\textwidth]{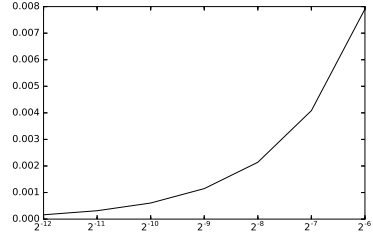}
         \caption{Holderian error with respect to step size $h$.}
     \end{subfigure}
        \caption{Comparison between $A^{\alpha}f$ and $A_{h}^{\alpha}f$ when $f(x)=x^{\mu}\ln x$ \\for $\mu=3/2,\alpha=0.3$ and  $\beta=0.1$.}
    \label{fig}
\end{figure}

\newpage
\subsection{Exemple 2}

For the second example we consider the fractional differential equation
presented in \cite{dietford}, for $t\in\left[  0,1\right]$.

\begin{equation}
D^{\alpha}y(t)=\frac{40320}{\Gamma(9-\alpha)}t^{8-\alpha}-3\frac
{\Gamma(5+\alpha/2)}{\Gamma(5-\alpha/2)}t^{4-\frac{\alpha}{2}}+\frac{9}%
{4}\Gamma(\alpha+1)+\left(  \frac{3}{2}t^{\frac{\alpha}{2}}-t^{4}\right)
^{3}-\left[  y(t)\right]  ^{\frac{3}{2}}
\label{exemple2}%
\end{equation}

The initial condition is $y(0)=0$.
The exact solution of this problem is \[y(t)=t^{8}-3t^{4+\frac{\alpha}{2}%
}+\frac{9}{4}t^{\alpha}.\]
\\For $\alpha=0.5,$ we display the results in table \ref{table2}  for $\beta=0.1$ and $\beta=0.01$ respectively.
\\Apparently, we need to use small  values for $\beta$ to increase the accuracy.

\begin{table}[h]
\centering
\begin{tabular}
[c]{|l|l|l|}\hline
$h$ & Errors for $\beta=0.1$ & Errors for $\beta=0.01$\\\hline
$2^{-7}$ & 0.0347581 & 0.0224598\\\hline
$2^{-8}$ & 0.0269360 & 0.0163528\\\hline
$2^{-9}$ & 0.0206910 & 0.0118018\\\hline
$2^{-10}$ & 0.0158085 &  0.0084716\\\hline
$2^{-11}$ & 0.0120388 & 0.0060613\\\hline
$2^{-12}$ & 0.0091502 & 0.0043283 \\\hline
$2^{-13}$ & 0.0069465 & 0.0030872\\\hline
\end{tabular}

\caption{ H\"{o}lderian errors for problem (\ref{exemple2}) with $\alpha=0.5$  }

\label{table2} 

\end{table}

\section{Conclusion}
\qquad In this work, we defined a fractional operator as a fractional power of a piecewise linear interpolation of a backward difference on a H\"{o}lder space. We proved the strong convergence  of this operator to fractional derivative, and we supported our results with examples. We think that we have now a kind of process to define  Euler-like formulas which contribute to solve numerically fractional differential equations in H\"{o}lder spaces.
 However, several questions can be the subject of further works, in particular, the analysis of the order of approximation, and the results that can be expected if  one replace the linear spline
$\mathcal{I}_{h}$ by a spline of degree $n>1$, or if one replaces the  operator  $\nabla_{h}$  by another more accurate approximation of the derivative.


\begin{thebibliography}{00}

%
\bibitem {abdelj2011}T. Abdeljawad, On Riemann and Caputo fractional
differences, Computers and Mathematics with Applications 62 (2011) 1602--1611, doi:10.1016/j.camwa.2011.03.036

\bibitem {abdelj2012}T. Abdeljawad and F. M. Atici,On the Definitions of
Nabla Fractional Operators,Abstract and Applied Analysis, Volume 2012 (2012),
Article ID 406757, 13 pages http://dx.doi.org/10.1155/2012/406757

\bibitem {abramowitz}M. Abramowitz, I.A. Stegun, Handbook of mathematical
functions with formulas, graphs and mathematical tables, Nat. Bur. Stands, 1964.

\bibitem {ash}A. Ashyralyev,A note on fractional derivatives and fractional
powers of operators,J. Math. Anal. Appl. 357 (2009) 232--236, doi:10.1016/j.jmaa.2009.04.012

\bibitem {das}S. Das, Functional fractional calculus, Springer, 2011

\bibitem {dietford}K. Diethelm, N. J. Ford and A. D. Freed, Detailed error
analysis for fractional Adam method, Numerical Algorithms 36: 31-52 (2004).

\bibitem {diethelm}K. Diethelm,The Analysis of Fractional Differential
Equations,An Application-Oriented Exposition Using Differential Operators of
Caputo Type ,Lecture Notes in Mathematics, Springer,2010

\bibitem {elezovic}Elezovi\'{c}, N., Lin, L., Vuk\v{s}i\'{c}, L.: Inequalities
and asymptotic expansions of the Wallis sequence and the sum of the Wallis
ratio. J. Math. Inequal. 7(4), 679--695 (2013)

\bibitem {garrappa}R. Garrappa, Numerical Solution of Fractional Differential
Equations: A Survey and a Software Tutorial, Mathematics 2018, 6(2), 16; doi:10.3390/math6020016

\bibitem {haase}M. Haase, The Functional Calculus for Sectorial Operators,
Operator Theory: Advances and Applications, Birkh\"{a}user, Basel, (2006).
ISBN-10: 3-7643-7697-X.

\bibitem {hilfer}Hilfer, R.: Applications of Fractional Calculus in Physics.
World Scientific Publishing Co, Singapore (2000)

\bibitem {ito}K. Ito, F. Kappel, Evolutions Equations and Approximations,
World Scientific, Singapore, (2002)



\bibitem {martinez}C. Martinez and M. Sanz,The Theory of Fractional Powers of
Operators, Volume 187,North Holland ,(2001).

\bibitem {mozy}D. Mozyrska and E. Girejko, Overview of Fractional h-difference
Operators, In: Operator Theory: Advances and Applications, vol. 229, pp.
253--267. Birkh\"{a}user (2013), doi:10.1007/978-3-0348-0516-2.

\bibitem {orteguiera}M. D. Ortigueira, Fractional Calculus for Scientists and
Engineers,Lecture Notes in Electrical Engineering, Volume 84, Springer 2011

\bibitem {rackauska}A Ra\v{c}kauskas and Charles Suquet, Functional Laws of
Large Numbers in H\"{o}lder Spaces, ALEA, Lat. Am. J. Probab. Math. Stat. 10
(2), 609--624 (2013)

\bibitem {samko93}S.G. Samko, A.A. Kilbas and O.I. Marichev, Fractional
integrals and derivatives, Theory and Applications, Gordon and Breach Science
Publishers (1993)

\bibitem {samko94}S.G. Samko and Z.U. Mussalaeva, Fractional type operators in
weighted generalized H\"{o}lder spaces,Georgian Mathematical Journal 1(1994),
No. 5, 537-559.

\bibitem {valerio}D. Val\'{e}rio, J.J. Trujillo, M. Rivero, J.A.T. Machado and
D.Baleanu, Fractional calculus: A survey of useful formulas, Eur. Phys. J.
Special Topics 222 ,1827--1846(2013) DOI:10.1140/epjst/e2013-01967-y.

\bibitem {wendel}Wendel, J.G.: Note on the gamma function. Am. Math. Mon. 55,
563--564 (1948)

\bibitem {white}H. E. White, Jr, Functions with a Concave Modulus of
Continuity, Proceedings of the American Mathematical Society, Vol. 42, No. 1
(Jan., 1974), pp. 104-112.

\bibitem {zhoo}Y. Zhou, Basic Theory of Fractional Differential Equations,
World Scientific Publishing Co. Pte. Ltd. 2014 \end{thebibliography}
\end{document}